\documentclass{amsart}
\usepackage{graphicx}
\usepackage{psfrag, subfigure}
 \usepackage{ amssymb }
  \usepackage{ amssymb }
\newtheorem{theorem}{Theorem}[section]
\newtheorem{lemma}[theorem]{Lemma}
\newtheorem{prop}{Proposition}
\theoremstyle{definition}
\newtheorem{definition}[theorem]{Definition}

\newtheorem{rk}{Remark}
\theoremstyle{remark}

\newcommand{\N}{\mbox{$\mathbb{N}$}}

\numberwithin{equation}{section}

\begin{document}

\title{Examples of minimal set for IFSs.}

\author{N. Guelman, J.Iglesias and A.Portela.}

\address{N. Guelman, Universidad de La Rep\'ublica. Facultad de Ingenieria. IMERL. Julio
Herrera y Reissig 565. C.P. 11300. Montevideo, Uruguay }
\email{nguelman@fing.edu.uy}

\address{J. Iglesias, Universidad de La Rep\'ublica. Facultad de Ingenieria. IMERL. Julio
Herrera y Reissig 565. C.P. 11300. Montevideo, Uruguay}
\email{jorgei@fing.edu.uy }

\address{A. Portela, Universidad de La Rep\'ublica. Facultad de Ingenieria. IMERL. Julio
Herrera y Reissig 565. C.P. 11300. Montevideo, Uruguay }
\email{aldo@fing.edu.uy }

\begin{abstract}
We exhibit different examples of minimal sets for an IFS of homeomorphisms with rotation number equal to 0. It is proved that these examples are, from a topological point of view, the unique possible cases.

\end{abstract}
\maketitle

\section{Introduction.}

Let $f:S^{1}\to S^{1}$ be a homeomorphism. A minimal set for $f$ is an $f$-invariant closed set that is minimal (in the sense of inclusion).

Taking into account the rotation number of a homeomorphism $f:S^{1}\to S^{1}$
there exist three possibilities:
\begin{enumerate}
\item $f $ has a finite orbit i.e. there exists a finite minimal set,
\item $ S^1$ is minimal for $f $,
\item there is a unique minimal set that is a Cantor set.
\end{enumerate}

This trichotomy can be extended to actions of groups of
homeomorphisms on the circle (see for example \cite{n}):

 Let $G$ be a subgroup of homeomorphisms of $S^1$. Then
one of these three possibilities occur:

\begin{enumerate}
\item there exists a finite orbit of $G$,
\item $ S^1$ is minimal for $G $,
\item there is a unique minimal set that is a Cantor set.
\end{enumerate}

The aim of this paper is describe any minimal set for the action of a finitely generated semigroup.

Given n maps $f_1,...,f_n$, $f_i:S^{1}\to S^{1}$, let
 $IFS(f_1,...,f_n)$ be the set of finite composition of  $f_i$, that is any finite "word" of $f_1,....,f_n$. For any $x\in S^{1}$ we define orbit of  $x$ as $O(x)=\{\phi(x): \ \phi\in
IFS(f_1,...,f_n)       \}$.
A set  $M\subset S^{1}$ is
invariant if $f_i(M)\subset M$ for any $i=1,...,n$. It is obvious that $\phi(M)\subset M$ for any $ \phi\in
IFS(f_1,...,f_n)$. A non- empty set  $M$ is minimal, when it is closed, invariant and $\overline{O(x)}=M$ for any $x\in M$.

An $IFS$ can have infinite many minimal sets and it is possible that $\phi^{-1}(M)\neq M$ for an element of the family. This is a big difference with the case of a group action. There, any minimal set is invariant for the future and for the past.
The aim of this paper is to describe the  minimal sets of $IFS$, studied from a topological point of view, and to give examples of any possible case. We will construct examples where the rotation number of any generator is 0, so that some  examples could be $C^{1}$-robust.
Several authors have
constructed examples of minimal sets of $IFS$. In \cite{br} sufficient conditions are given for an $IFS$ generated by two elements that are close to rotation so that $S^{1}$ is a minimal set.
In fact, Duminy (see \cite{n}, section 3.3) has proved that if $f,g$ are diffeomorphisms $C^2$-close to identity, the minimal set for $IFS(f,g)$ is $S^1$.

In \cite{s} sufficient conditions are given for an $IFS$ generated by two elements with overlap and rotation number 0 so that the minimal set is a Cantor set.
 These construction are important tools for the construction of our examples.

 If $M\subset S^{1}$ is a closed set, we call $\mathcal{I}$ the set of connected components of the interior of $M$, $int(M)$, and  $K$ the maximal Cantor set included in  $M\setminus int (M)$. We will prove that we can write $M$ as the disjoint union of $ \mathcal{I}$, $ K$ and $ N$, where $N$ is a countable set.

In \cite{br} is proven a trichotomy for a minimal set of an $IFS$: it is finite, it is a Cantor set or it has non-empty interior. We will improve this result proving the following

\begin{theorem}
\label{theor}
Let $f_1,...,f_n$ be open functions, $f_i:I\to I$ where $I$ is a closed interval or $S^{1}$ and $M=\mathcal{I}\cup K\cup N$ is minimal set for the family $IFS(f_1,...,f_n)$.
The possibilities for M are:

\begin{enumerate}

 \item If $int(M)=\emptyset$  then
M is finite or Cantor set.

 \item If $int(M)\neq\emptyset$  then
 \begin{itemize}
 \item M is $S^1$,
 \item $M= \mathcal{I}\cup K\cup N$ with     $ \partial \mathcal{I}\cap N\neq\emptyset$,
 \item  $M= \mathcal{I}\cup N$, or
 \item $M= \mathcal{I}\cup K$.
 \end{itemize}

\end{enumerate}
\end{theorem}
From this theorem we have that unique cases for $ \mathcal {I}, K $ and $ N $ that are not possible are the following:
\begin{enumerate}

 \item $K= \mathcal{I}=\emptyset$ and $N$ is an infinite set.

 \item $\mathcal{I}=\emptyset$, $K\neq\emptyset$ and $N\neq\emptyset$.

  \item $N\neq\emptyset$,   $K\neq \emptyset $ and  $\partial \mathcal{I}\cap N=\emptyset$.

\end{enumerate}

Note that a particular case of $M= \mathcal{I}\cup N$, is realized when $M$ is an union of disjoint intervals (finite or infinite).

We will exhibit examples (see  section \ref{examples}) to show that any case of  Theorem \ref{theor} can be realized and we will prove that the remaining cases are not possible.



When $M= \mathcal{I}\cup K$ (see example 7), then $M$ is a  symmetric Cantorval (see definition \ref{cantdef}). This sets were studied in \cite{GN} and \cite{MO} and we use definition given by Z. Nitecki (see \cite{Ni}). We prove in Proposition \ref{cantorvals2} that this is the unique case where  $M$ is a symmetric Cantorval.

\medskip

Examples given by Denjoy show that it is possible that a Cantor set is a minimal set for a diffeomorphism. Then, it is a natural question if any Cantor set is minimal for a diffeomorphism. As any two Cantor sets  are "$C^{0}$-homeomorphic" then by conjugating a Denjoy's example, it is easy to prove that any Cantor set is a minimal set for a $C^{0}$-homeomophism. A homeomorphism of class $C^{2}$ can not have a Cantor minimal  set. Herman \cite{H}, McDuff \cite{mc}, Kercheval \cite{N}, Kra \cite{KS} and Portela \cite{P} proved that there exist Cantor sets ( for example the usual triadic set) such that are not minimal sets for  $C^{1}$  diffeomorphism. We will see in Subsection \ref{subsec1}(example (1)) ,  that the  usual triadic set is minimal for $IFS(f,g)$ where $f$ y $g$ are  $C^{\infty}$, then it is natural to ask if  any Cantor set is  minimal for an $ IFS$ where any element is of class  $C^{1}$?\\
A more general question is the following:

Given a closed set $M=\mathcal{I}\cup K\cup N$ different of (1), (2) and (3). Are there open functions  $f_1,...,f_n$, $f_i:S^{1}\to S^{1}$, such that $M$
  is a minimal set for the family $IFS(f_1,...,f_n)$?

  In \cite{s1} K. Shinohara exhibited an example of an $IFS(f,g)$, where $f,g:I \rightarrow I$ are $C^1$-close to identity, with minimal set different of $I$. A natural question is in which of the possibilities of Theorem 1 is the minimal set of Shinohara's example.

\section{Topological Classification of closed sets.}

In this section we will give a topological description of closed sets and we will prove  lemmas that will be used in the proof of the main Theorem

Recall that a set is called perfect if it is closed and any point is an accumulate one.
Note that any perfect set is uncountable.
Given a set  $M\subset S^{1}$, we denote  $M^{'}$ the set of accumulate points.

\begin{lemma}\label{l5}
  Let $M\subset S^{1}$, where $M$ is closed,  totally disconnected and uncountable, then there exists $K\subset M$, $K$ a maximal Cantor set such that $M\setminus K$ is countable.

\end{lemma}

\begin{proof}
 Let $B(x,\varepsilon )$ the ball of center $x$ and  radius $\varepsilon >0$. Let  $$K=\{x\in M : \ \forall \varepsilon >0 \mbox{ it holds that }  B(x,\varepsilon)\cap M \mbox{  is uncountable } \}.$$
It is easy to see that $K$ is closed and  totally disconnected since  $M$ is. To prove that $K$ is a Cantor we have to prove that $K^{'}=K$. For that, it is enough to show that $M\setminus K$ is countable. Let $y\in M\setminus K$ then there exists $\varepsilon>0$ such that $B(y,\varepsilon)\cap M $ is countable. Let $U_y$ be an open interval with  rational end points and $y\in U_y\subset B(y,\varepsilon)$ then $$M\setminus K=\bigcup_ {y\in M\setminus K}U_y\cap M.$$
Since the number of $U_y$ is countable and $U_y\cap M$ is countable, it follows that  $M\setminus K$ is countable.
$K$ is maximal since if there exists a Cantor set  $K_1$, with $K\subsetneq K_1$ then $K_1\setminus K$ is uncountable, therefore $M\setminus K$ is uncountable, which is a contradiction.
\end{proof}

\begin{lemma}\label{principal}
Let $M\subset S^{1}$ be a closed set
then $M$ can be written in a unique way as a disjoint union of  $\mathcal{I}$, $K$ and $N$ where
$\mathcal{I}$ is the union of the connected components of $int (M)$, $K$ is a maximal
Cantor set in $M\setminus \mathcal{I}$ and $N$ is a countable set.
\end{lemma}
\begin{proof}
 In the case that $int(M)=\emptyset$ we have two possibilities: If  $M$ is countable, the statement is obvious; if $M$ is uncountable, by Lemma \ref{l5}, we have that $M=K\cup N$.\\
In the case that $int(M)\neq\emptyset$, let $\mathcal{I}$ be the  union of the interior of the non-trivial connected components of $M$. Then $M\setminus \mathcal{I}$ is closed and  $int (M\setminus \mathcal{I})=\emptyset$. Therefore by the previous case we have that $M=\mathcal{I}\cup K\cup N$, so we have showed the existence.
The unicity is a consequence of the definition of $\mathcal{I}$ and because $K$ is maximal in  $M\setminus \mathcal{I}$.
\end{proof}

\begin{rk}\label{obs1}

  Since $M$ is closed it follows that $\overline{\mathcal{I}}\subset M$ then $\partial \mathcal{I}\subset K\cup N$.

\end{rk}

\medskip
 From now on, we consider $IFS(f_1,...,f_n)$ where $f_i:I\to I$ are open function, $I$ is an open interval or $S^{1}$ and $M$ is a minimal set for the  family.

Next lemma is a consequence of the fact that any perfect set is uncountable.

\begin{lemma}\label{l1}
  Let $M$ be a minimal set with $\sharp M=\infty$. Then $M$ is uncountable.
\end{lemma}
\begin{proof}
   If $M^{'}=\emptyset$ then $M$ is finite. Therefore  $M^{'}\neq\emptyset$. Since $f_i:I\to I$ are open functions then $f_i(M^{'})\subset M^{'}$. Therefore, $M=M^{'}$ by minimality of $M$, so  $M$ is a perfect set and it follows that it is uncountable.
\end{proof}

Then, the case $\mathcal{I}= K=\emptyset$ and $N$ with infinite many points is not possible.

\begin{lemma}\label{l2}
  Suppose that a minimal set has finitely many non-trivial connected  components, then the minimal set is such a union.

\end{lemma}
\begin{proof}
 Let $I_1,...,I_n$  be the non-trivial connected  components of $M$. Since any $f_i$ is open, then $f_i(I_j)$ is an interval for any $i$ and $j$. Therefore, $f_i(I_j)\subset \cup I_k$. It follows that $M=\cup I_k.$

\end{proof}
Let us introduce the following

\begin{definition}
\label{cantdef}
A symmetric Cantorval is a nonempty compact subset $M$ of $S^1$ (or the
real line) such that:\\
(1) $M$ is the closure of its interior (i.e., the nontrivial components are dense)\\
(2) Both endpoints of any nontrivial component of $M$ are accumulation points
of trivial (i.e., one-point) components of $M$.
\end{definition}
\begin{rk}\label{obs2}
If $M= \mathcal{I}\cup K\cup N$ is a  symmetric Cantorval, then for all $z\in K\cup N$ there exists $\{x_n\}\subset K\cup N$, $x_n\neq z$ such that $x_n\to z$.
\end{rk}

Example 7 of section \ref{examples} shows that a  minimal  set can be a symmetric Cantorval.

\begin{lemma}\label{cantorvals}
Let $M= \mathcal{I}\cup K\cup N$. If $M$ is a  symmetric Cantorval then $N=\emptyset$.
\begin{proof}

Suppose that $N\neq\emptyset$. Let $x_0\in N$ and $\varepsilon >0$ be such that $\overline{B(x_0, \varepsilon )}\cap K=\emptyset$. We will  prove that $\overline{B(x_0, \varepsilon )}\cap N$ is a perfect set. It is obvious that it is closed. Given $z\in \overline{B(x_0, \varepsilon )}\cap N$ as $M$ is a  symmetric Cantorval , by Remark \ref{obs2}, there exist $\{x_n\}\subset (K\cup N)\cap\overline{B(x_0, \varepsilon )} $, $x_n\neq z$ such that $x_n\to z$.
As $\overline{B(x_0, \varepsilon )}\cap K=\emptyset$ then $x_n\in N$, so $\overline{B(x_0, \varepsilon )}\cap N$ is a perfect set. Therefore it is  uncountable and this is a contradiction.
\end{proof}

\end{lemma}

\begin{prop}\label{cantorvals2}
Let $IFS(f_1,...,f_n)$ where $f_i:I\to I$ are open function, $I$ is an open interval or $S^{1}$ and $M$ is a minimal set for the  family. Then\\

$M$ is a  symmetric Cantorval if and only if $M=\mathcal{I}\cup K$.
\end{prop}
\begin{proof}
($\Rightarrow$) If $M= \mathcal{I}\cup K\cup N$, by Lemma \ref{cantorvals} we have that $N=\emptyset$.\\
($\Leftarrow$)  Given $z\in M$ and $I_0\in \mathcal{I}$ as $M$ is a minimal set there exist $\{ \phi_n \}\subset IFS(f_1,...,f_n)$ such that $\phi_n(I_0)\to z$. Then $M$ is the closure of its interior.
As $\partial  \mathcal{I}\subset K$ and $K$ is a Cantor set then both endpoints of any nontrivial component of $M$ are accumulation points of trivial components of $M$.
\end{proof}

\section{First part of the Proof of Theorem \ref{theor}}

This proof has 2 steps. In the first one we will prove that the following cases are not possible. In the second one we will construct examples of the remaining cases. These examples will be shown in Section \ref{examples}.

As a  consequence of  the previous lemmas, the following cases can not occur:
\begin{itemize}

 \item $K= \mathcal{I}=\emptyset$ and  $N$ with infinite many points. Because of  $M$ would be countable and this fact is not possible by Lemma \ref{l1}.

 \item $\mathcal{I}=\emptyset$ and $K\neq\emptyset$, $N\neq\emptyset$.
  Since $\mathcal{I}=\emptyset$, any  $f_i$ is open and $K$ is a maximal Cantor set, then $f_i (K)\subset K$ for any $i$. So, for any  $\phi$ in IFS we have that $\phi (K)\subset K$. Therefore $M\subset K$, which is a contradiction.

  \item $M=K\cup \mathcal{I}\cup N$ where $\partial \mathcal{I}\subset K$. Since $\partial \mathcal{I}\subset K$ then $K\cup \mathcal{I}$ is closed, so $S^{1}\setminus( K\cup \mathcal{I})$ is open and $N\subset S^{1}\setminus( K\cup \mathcal{I})$. Let $T$ be a connected  component of $S^{1}\setminus( K\cup \mathcal{I})$ with $N\cap T\neq \emptyset$. Notice that $N\cap T$ is countable.
      On the other hand, for $x\in K$ there exists $\phi$ in IFS such that $\phi(x)\in T$. Then $\phi (K)\cap  T$ is uncountable, which is a contradiction.

\end{itemize}

\section{Tools for example's construction.}

\begin{figure}[h]
\psfrag{i0}{$I_0$}\psfrag{i1}{$I_1$}
\psfrag{i-1}{$I_{-1}$}
\psfrag{f}{$f$}
\psfrag{g}{$g$}
\psfrag{i}{$\widehat{I}$}
\psfrag{iii}{$I^{'}$}
\psfrag{t+}{$T^{+}_{I^{'}I}$}\psfrag{t-}{$T^{-}_{I^{'}I}$}
\psfrag{ii}{$I$}
\psfrag{a}{$a$}
\psfrag{b}{$b$}
\psfrag{2b}{$\frac{(a+b)}{3}$}
\psfrag{2a}{$\frac{2(a+b)}{3}$}\psfrag{i1}{$I_1$}
\psfrag{aaa}{$I_0$}
\begin{center}
\psfrag{z}{$I^{'}$}
\caption{\label{figura1}}

\subfigure[]{\includegraphics[scale=0.14]{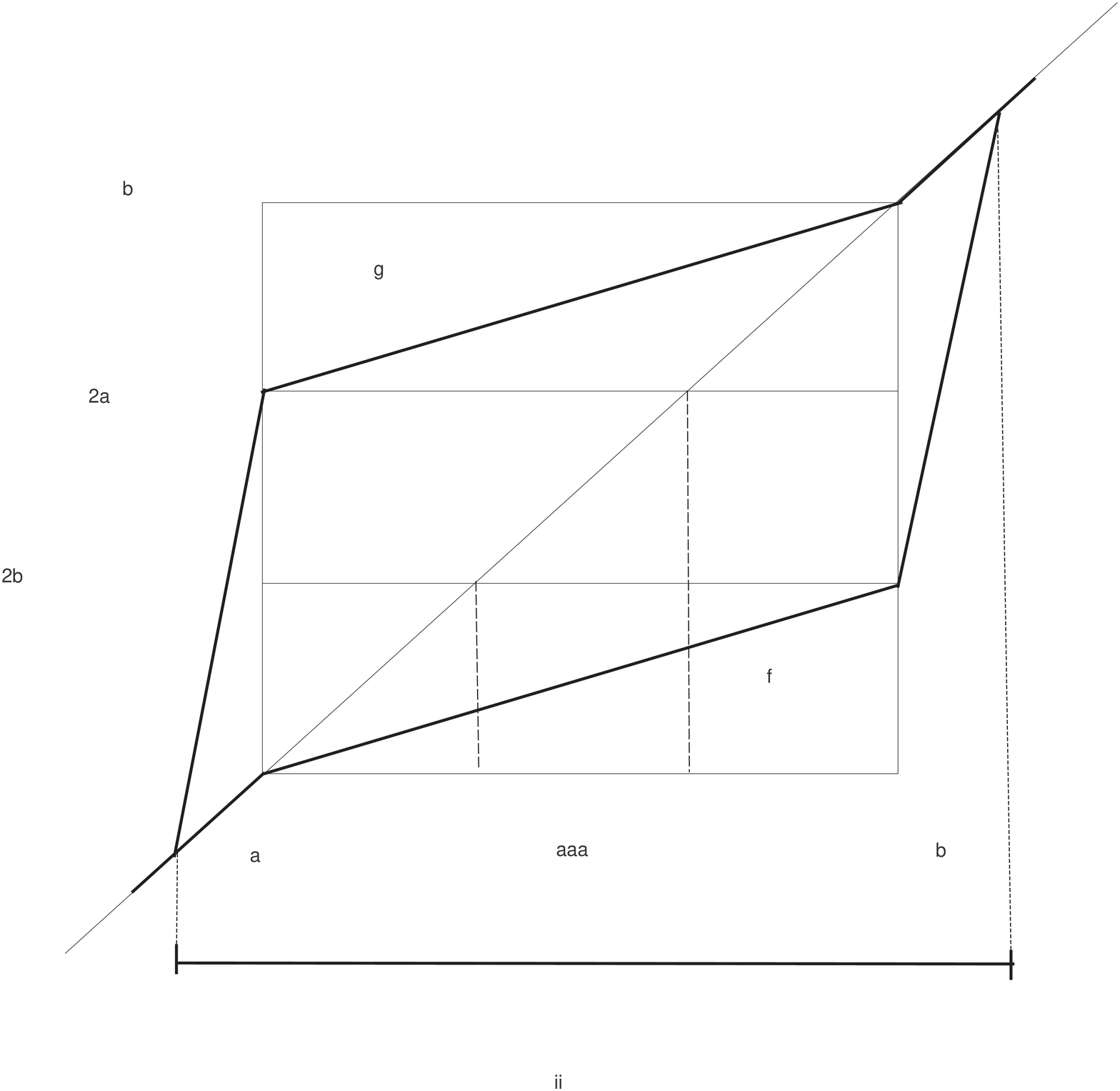}}
\subfigure[]{\includegraphics[scale=0.15]{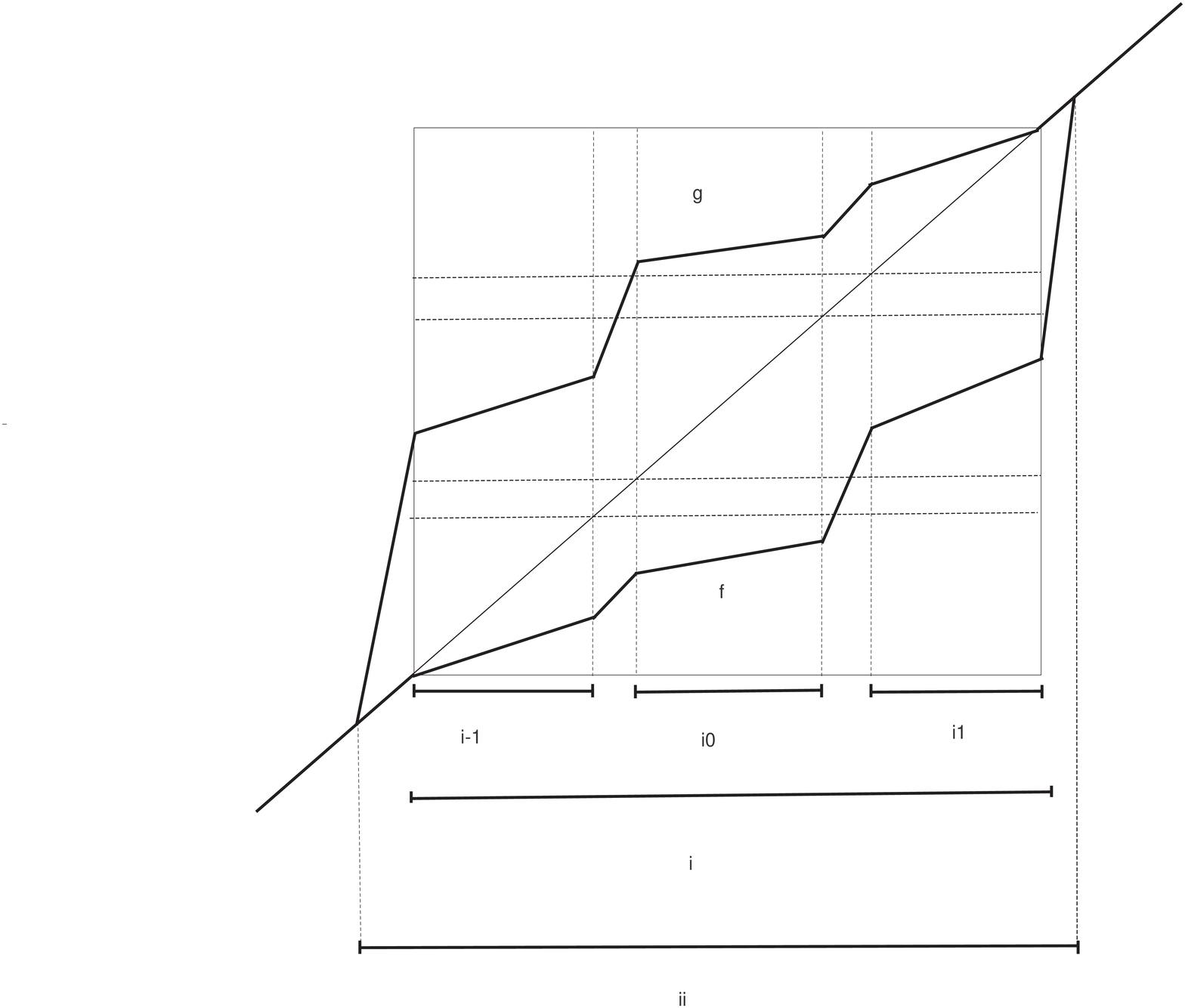}}
\end{center}
\end{figure}

We will use the following IFS and functions for the construction of the examples.

\subsection{Examples of IFS with Cantor  minimal set}\label{subsec1}

In this section we will construct IFS with Cantor minimal set.

\begin{enumerate}
 \item

 Given an interval $I\subset S^{1}$, we consider an interval $I^{'}=[a,b]$ with $I^{'}\subsetneq I$. Let  $f$ and $g$ be linear maps of angular coefficient  1/3 in $I^{'}$, as  Figure \ref{figura1}(a) and such that $f|_{I^{c}}=g|_{I^{c}}=Id$.
 If $\Lambda_0=[a,b]$ and $\Lambda_{k+1}=f(\Lambda_k)\cup g(\Lambda_k)$ then $K=\cap \Lambda_k$ is a Cantor set ( the usual triadic set). It is easy to see that $K$ is a minimal set for $IFS(f,g)$.\\
Later we will use the following properties:
\begin{enumerate}
\item $f^{-1}(K)\cap I'\subset K$ and $g^{-1}(K)\cap I'\subset  K$.
\item If $x\in I^{'}$ then $\overline{O(x)}\supset K.$

 \end{enumerate}

  \item Next example is in \cite{s}. Given two intervals $\widehat{I}$ and $I$, with $\widehat{I}\subsetneq I$, we   consider  intervals $I_{-1},I_0$,  $I_1$ and
 the maps $f$ y $g$ as in Figure \ref{figura1} (b). These maps verify that:
\begin{itemize}
 \item $f(I_{-1})\subset I_{-1}$,  $f(I_{0})\subset int( I_{-1})$, $f(I_{1})\subset int(I_{0})$.

  \item $g(I_{1})\subset I_{1}$,  $g(I_{0})\subset int( I_{1})$, $g(I_{-1})\subset int(I_{0})$.

  \item $f^{'}|_{ I_{-1}\cup I_{0}\cup I_{1}  },g^{'}|_{ I_{-1}\cup I_{0}\cup I_{1}  }<\lambda<1.$

  \item $f|_{I^{c}}=g|_{I^{c}}=Id.$

\end{itemize}

 If $\Lambda_0:= I_{-1}\cup I_{0}\cup I_{1}$ and $\Lambda_{k+1}:=f(\Lambda_k)\cup g(\Lambda_k)$ then  $K=\cap_{k\geq 0} \Lambda_k$  is a Cantor set and it is  minimal  for $IFS(f,g)$.

As in the Example (1), following properties are satisfied:

\begin{enumerate}
\item $f^{-1}(K)\cap \widehat{I}\subset K$ and $g^{-1}(K)\cap \widehat{I}\subset K$.
 \item If $x\in \widehat{I}$ then $\overline{O(x)}\supset K.$
\end{enumerate}

From now on, we write $\{ f_{_{K,I}},g_{_{K,I}} \}$ when  IFS generated by  $f$ and $g$ has minimal set $K\subset I$, which is a  Cantor set,  $f$ and $g$ are the  identity map in $I^{c}$ and verifies properties (a) and (b) as above.

\item Let $IFS(f)$,  where $f$ is a Denjoy diffeomorphism.

\end{enumerate}

\subsection{Examples of IFS with an interval minimal set}\label{subsec2}

Next  example is an IFS whose minimal set is an interval and each generator has 0  rotation number.\\
Construction and properties of the following example are in \cite{br}. Let $I^{' }=[a,b]$ and $I$ be two intervals with $I^{'}\subsetneq I$, we consider  $T^{+}_{I^{'}I}$ and $T^{-}_{I^{'}I}$ be linear maps in $I^{'}$ (see Figure \ref{tmas} (a)) and
 $T^{+}_{I^{'},I}$ and $T^{-}_{I^{'},I}$ are the identity map on $I^{c}$. The minimal set of $IFS(T^{+}_{I^{'}I},T^{-}_{I^{'}I})$ is the interval  $I^{'}$. Let us denote  $f=T^{-}_{I^{'}I}$ and $g=T^{+}_{I^{'}I}$; for any $x\in [a,b]$ we have that $f^{n}(x)\to a$ and $g^{n}(x)\to b$. Then, if $M\subset [a,b]$ is a minimal set then $\overline{O(a)}= \overline{O (b)}=M$. Suppose that  $\overline{O (a)}\neq [a,b]$, let $I_0$ be the connected  component of $[a,b]\setminus \overline{O (a)}$ with maximum length. Since $f([a,b])\cup g([a,b])=[a,b]$ then $f^{-1}(I_0)\cap [a,b]\neq\emptyset$ or $g^{-1}(I_0)\cap [a,b]\neq\emptyset$. In the case that $f^{-1}(I_0)\cap [a,b]\neq\emptyset$, if $f^{-1}(I_0)\subset [a,b]$ as $|f^{'}|<1$ then $|f^{-1}(I_0)|> |I_0|$ which is a contradiction. Therefore $f^{-1}(I_0)$ contains $a$ or $b$, hence $I_0$ contains  $a$ or $b$ which is also a contradiction. The case  $g^{-1}(I_0)\cap [a,b]\neq\emptyset$ is analogous.

\begin{figure}[h]
\psfrag{t+}{$T^{+}_{I^{'}I}$}\psfrag{t-}{$T^{-}_{I^{'}I}$}
\psfrag{iii}{$I^{'}$}
\psfrag{z}{$I$}
\psfrag{a}{$a$}\psfrag{b}{$b$}

\begin{center}
\caption{\label{tmas}}
\subfigure[]{\includegraphics[scale=0.24]{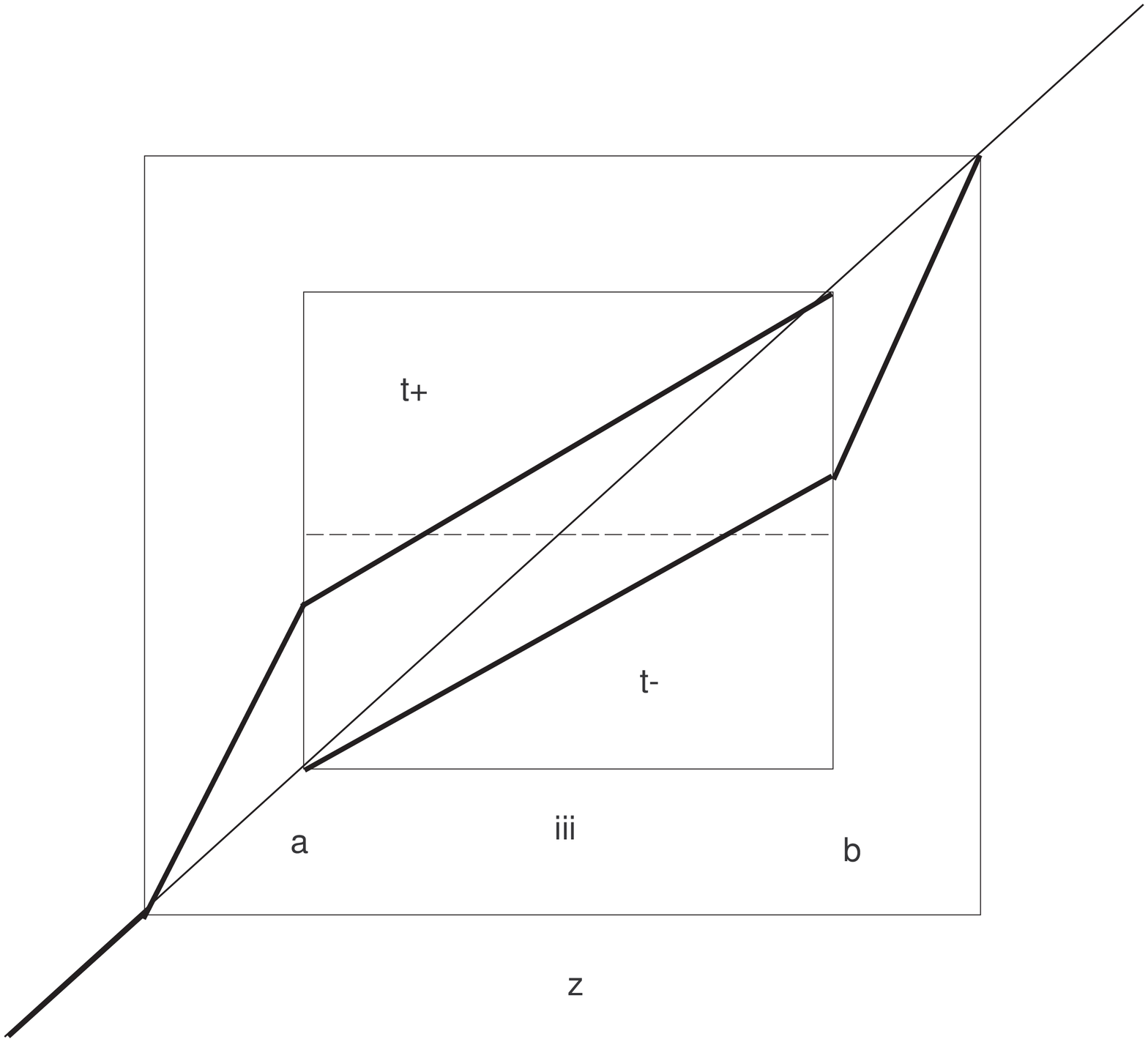}}
\end{center}
\end{figure}

Notice that for a big enough $n\in \N$, $IFS((T^{+})^{n}_{I^{'}I},(T^{-})^{n}_{I^{'}I})$ is in similar conditions than example (1) of section \ref{subsec1}, therefore it has a Cantor  minimal set. It is obvious that any minimal set of  $IFS((T^{+})_{I^{'}I},(T^{-})_{I^{'}I})$ contains a minimal set of $IFS((T^{+})^{n}_{I^{'}I},(T^{-})^{n}_{I^{'}I})$. Therefore the minimal set of $IFS((T^{+})^{n}_{I^{'}I},(T^{-})^{n}_{I^{'}I})$ and  $IFS((T^{+})^{}_{I^{'}I},(T^{-})^{}_{I^{'}I})$ are different from a topological point of view.

\subsection{Map to consider.}\label{subsec3}

   Given three intervals $J, J^{'}$ and $I$ with $J^{'}\subsetneq J$ and $J\cap I=\emptyset$. Let $h$ be a  homeomorphism such that $h|_{J^{'}}=Id$  $h(J^{c})\subset I$.\\ We will denote $h$ by $h_{J^{'},J,I}$.
\vspace{.5cm}

\section{Example's Construction.}\label{examples}

In this section we construct seven examples to prove Theorem \ref{theor}.

{\bf{Example 1}}. $M=K$ (minimal set that is a Cantor set) or $M$ is a finite set. We consider the following examples:
\begin{enumerate}
\item Let  $IFS( f_{_{K,I}},g_{_{K,I}} )$ as in  subsection \ref{subsec1} examples (1) y (2).
\item Let $g$ be a Denjoy homeomorphism and we consider  $IFS( g )$.
\end{enumerate}

It is clear that both examples have a Cantor minimal set.
The difference between (1) y (2) is that example (1) can be constructed $C^{\infty}$ and the generators of the IFS have 0 rotation number. In (2) the map $g$ can be  $C^{1+\alpha}$ with $\alpha <1$ and the rotation number of $g$ is irrational. In (2) the minimal is unique but this is not true for example (1).

\medskip

Let $f$ be a homeomorphism with fixed points. It is obvious that the family $IFS( f )$ has a finite minimal set.

These are examples if item (1) of Theorem \ref{theor}.\\

Examples with $int(M)\neq\emptyset$.

{\bf{Example 2.}}  $M=S^{1}$.\\
It is easy to find examples where the generators have irrational rotation number, but our aim is to exhibit examples of IFS with generators having 0 rotation numbers.

 Let $I\subset S^{1}$ be an interval, and let  $I^{'}$ and $J$ intervals such that $J\subset I^{'}\subsetneq I$. Let $H_1: S^{1}\to S^{1}$ be a homeomorphism such that $S^{1}\setminus I^{'}\subset H_1(J)$ and $H_2: S^{1}\to S^{1}$  such that $H_2(S^{1}\setminus I^{'})\subset J$.  We consider $$ IFS (T^{+}_{I^{'},I}, \ T^{-}_{I^{'},I} , \ H_1 , \ H_2 ). $$ The minimal set $M$ for this IFS is $M=S^{1}$. To show this, let $x\in S^{1}$. If $x\in I^{'}$ then $\overline{O(x)}\supset I^{'}$. Hence, $H_2(\overline{O(x)})\supset H_2(I^{'})\supset J^{c}$. Therefore $\overline{O(x)}=S^{1}$. If $x\notin I^{'}$ then $H_2(x)\in I^{'}$. From the previous case we have $\overline{O(x)}=S^{1}$.

{\bf{Example 3.}} $M=\mathcal{I}\cup K\cup N$ (a Cantor set union intervals union countable points). In this case we have $\partial \mathcal{I}=N$.

 Let $I_0\subset S^{1}$ be an interval. We consider IFS $( f_{_{K,I_0}},g_{_{K,I_0}} )$. Let $[a,b] \subset I_0$ be a minimal  interval containing a $K$,  let $I$ be a connected component of  $K^{c}$ with $I\subset [a,b]$ and $I^{'}\subset I$. Let $J\subset S^{1}\setminus I_0$ and $J^{'}\subset J$.
 We consider $$IFS( f_{_{K,I_0}},\ \ g_{_{K,I_0}},\ \  h_{J^{'},J,I^{'}},\ \ T^{+}_{I^{'}I},\ \ T^{-}_{I^{'}I}).  $$

We will prove that the minimal set for this IFS is  $$M=K\cup  \{   \phi (I^{'}): \ \ \phi \in IFS(        f_{_{K,I_0}}, g_{_{K,I_0}})            \} \cup I^{'}.     $$

We begin by proving that $M$ is IFS-invariant.
Let $x\in M$. If $x\in K$, then $f_{_{K,I_0}}(x)\in K$ and  $g_{_{K,I_0}}(x)\in K$. As $K\subset I_0$ then $x\notin J$ so $h_{J^{'},J,I^{'}}(x)\in I^{'}$. Finally, since $T^{+}_{I^{'}I}$ and $ T^{-}_{I^{'}I}$ are the identity on $K$ then $T^{+}_{I^{'}I}(x)\in K$ and $ T^{-}_{I^{'}I}(x)\in K$.\\
 If $x\in I^{'}$, then $f_{_{K,I_0}}(x)\in f_{_{K,I_0}}(I^{'})\subset M$. The same occurs with $g_{_{K,I_0}}(x)$. By definition of $h_{J^{'},J,I^{'}}$ we have that   $h_{J^{'},J,I^{'}}(x)\in I^{'}$. Also by definition of $T^{+}_{I^{'}I}$ we have that $T^{+}_{I^{'}I}(x)\in I^{'}$ and the same occurs with $ T^{-}_{I^{'}I}$. If $x\in \phi (I^{'})$, the proof is analogous that in case $x\in I^{'}$.
 Now, we are going to prove that  any $x\in M$ has dense orbit in $M$. Let $y\in M$ and  $U$ a neighborhood  of $y$ we will show that there  exists $\phi \in IFS( f_{_{K,I_0}},\ \ g_{_{K,I_0}},\ \  h_{J^{'},J,I^{'}},\ \ T^{+}_{I^{'}I},\ \ T^{-}_{I^{'}I})  $ with $\phi (x)\in U$. First we take $x\in K$. If $y\in K$, it is clear that there exists $\phi \in IFS( f_{_{K,I_0}},\ \ g_{_{K,I_0}}) $ such that $\phi (x)\in U$. If $y\in I^{'}$, as $x\notin J$ then $ h_{J^{'},J,I^{'}}(x)\in I^{'}$, then there is $\phi_1\in IFS (T^{+}_{I^{'}I},\ \ T^{-}_{I^{'}I})$ such that $\phi_1\circ h_{J^{'},J,I^{'}}(x)\in U$. If $y\in \phi (I^{'})$ the proof is analogous.

 Now we take $x\in I^{'}$. If $y\in K$, then by property (b) of example (1) in Section \ref{subsec1}, there exists  $\phi \in IFS( f_{_{K,I_0}},\ \ g_{_{K,I_0}}) $ such that $\phi (x)\in U$. If $y\in  I^{'}$ then there is $\phi\in IFS (T^{+}_{I^{'}I},\ \ T^{-}_{I^{'}I})$ such that $\phi (x)\in U$. If $y\in \phi (I^{'})$, we consider $U^{'}=\phi ^{-1}(U)\cap I^{'}$. There is $\phi_1\in IFS (T^{+}_{I^{'}I},\ \ T^{-}_{I^{'}I})$ such that $\phi_1(x)\in U^{'}$, therefore $\phi \phi_1(x)\in U$.

  If $x\in \phi (I^{'})$, then $ h_{J^{'},J,I^{'}}(x)\in I^{'}$, and the proof is the same that case $x\in I^{'}$.

{\bf{Example 4.}} $M=K\cup \mathcal{I}\cup N   $    (a Cantor set union intervals union countable points). In this case we have $N \nsubseteq \partial \mathcal{I}$.

 Let $I_0\subset S^{1}$ be an interval. We consider IFS $( f_{_{K,I_0}},g_{_{K,I_0}} )$. Let $[\alpha ,\beta ] \subset I_0$ be a minimal  interval containing a $K$.

Let $I$ be a connected component of  $K^{c}$ with $I\subset [\alpha ,\beta ]$ and $I^{'}\subset I$.

Let $I^{'}=[a,b]$, and $\varphi$ as in Figure \ref{varphi}, with $\varphi|_{(I^{'})^{c}}=Id$.

Let $I_1\subset [a,b]$  such that $\varphi ^{n}(I_1)\cap \overline{I_1}=\emptyset$ for any $n>0$. Notice that $\varphi ^{n}(I_1)\to b$.

Let  $\widehat{I}_1$ such that $I_1\subset \widehat{I}_1$ and $\varphi (I_1)\cap\overline{ \widehat{I}_1}= \emptyset$.

 Let $J\subset S^{1}\setminus I_0$ and $J^{'}\subset J$.

 We consider  $$IFS( f_{_{K,I_0}},\ \ g_{_{K,I_0}},  \ \ \varphi,        \ \  h_{J^{'},J,I_{1}},   \ \ T^{+}_{I_{1}\widehat{I_1}}, \ \ T^{-}_{I_{1}\widehat{I_1}}
   ) .  $$

Let us denote $\mathcal{F}= IFS(        f_{_{K,I_0}}, g_{_{K,I_0}}) $

Then the minimal set  is  $$M=K\cup \bigcup_{n\geq 0}\varphi^{n}(I_1)      \cup  \bigcup_{\phi\in\mathcal{F} } \phi ( \bigcup_{n\geq 0}\varphi^{n}(I_1) )     \cup \bigcup_{\phi\in\mathcal{F} } \phi (b)\cup \{b\}.$$

The proof of minimality of $M$ is analogous of above examples.

\begin{figure}[h]
\psfrag{i}{$I^{'}$}\psfrag{j}{$J$}
\psfrag{varphi}{$\varphi$}
\psfrag{varphi2}{$\psi$}
\psfrag{p}{$p$}
\psfrag{a}{$a$}
\psfrag{b}{$b$}
\psfrag{aa}{$0$}
\psfrag{bb}{$1$}
\psfrag{i0}{$I_0$}\psfrag{i1}{$I_1$}
\psfrag{i2}{$I_{2}$}
\psfrag{ii}{$I$}
\begin{center}
\caption{\label{varphi}}
\subfigure[]{\includegraphics[scale=0.15]{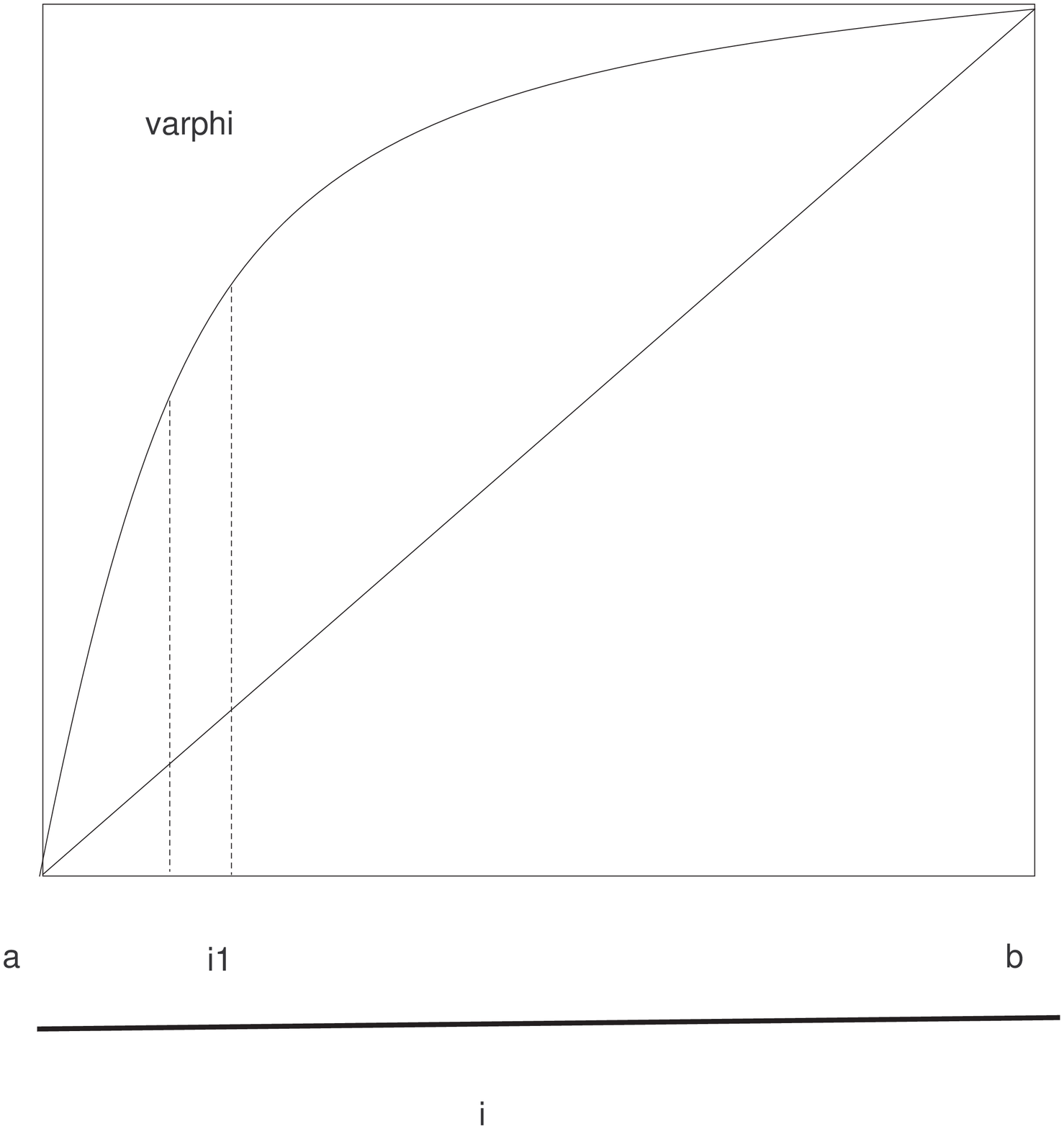}}
\subfigure[]{\includegraphics[scale=0.1535]{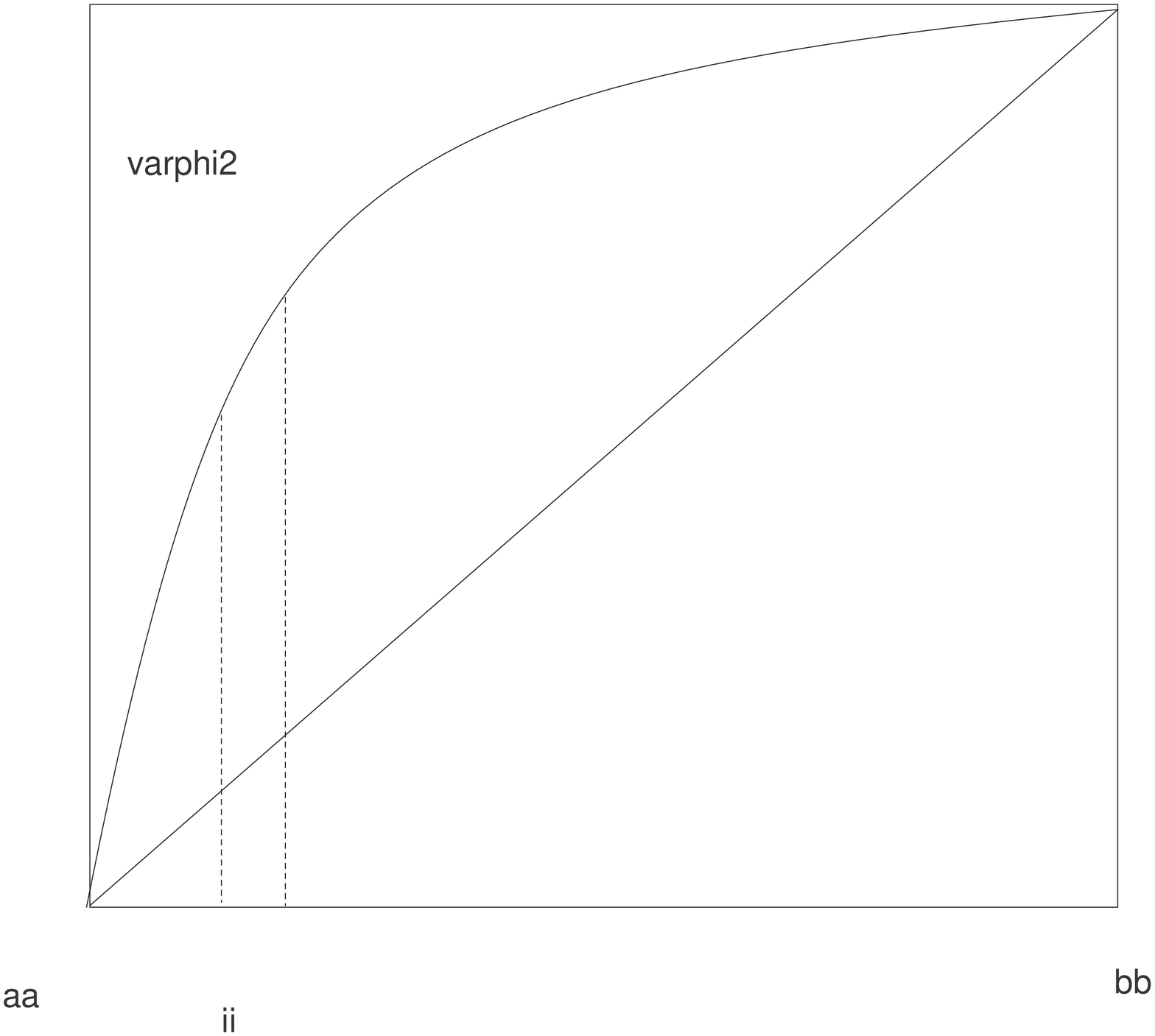}}
\end{center}
\end{figure}

{\bf{Example 5.}}  $M=K\cup \mathcal{I}\cup N   $ ( $M$ is an union of intervals.)\\
Let $I_0$, $I$ and $\psi$ as Figure \ref{intervalos} (a) and $I_1\supset I_0$ such that $\overline{\psi^{n}(I_1)}\cap \overline{I_1}=\emptyset$.
Let $J$ a interval such that $\overline{J}\subset S^{1}\setminus (\bigcup_{n\geq 0}  \{   \psi (I_1) \} \cup I)$ and $J^{'}\subset J$. We consider $h_{J^{'},J,I_{0}}$ and $H$ as Figure \ref{intervalos} (b). Note that $H(I_0)=I$ and $H( \bigcup_{n\geq 1}  \{   \psi^n (I_0) \} \cup I )\subset   \psi (I_0)$.

\begin{figure}[h]
\psfrag{i}{$I^{'}$}\psfrag{j}{$J$}
\psfrag{varphi}{$\psi$}
\psfrag{varphi2}{$\psi$}
\psfrag{h}{$H$}
\psfrag{a}{$a$}
\psfrag{b}{$b$}
\psfrag{aa}{$0$}
\psfrag{bb}{$1$}
\psfrag{i1}{$I_0$}
\psfrag{i2}{$I_{}$}
\psfrag{i3}{$I_0$}
\psfrag{ii}{$I_1$}
\psfrag{i4}{$\psi (I_{0})$}
\begin{center}
\caption{\label{intervalos}}
\subfigure[]{\includegraphics[scale=0.16]{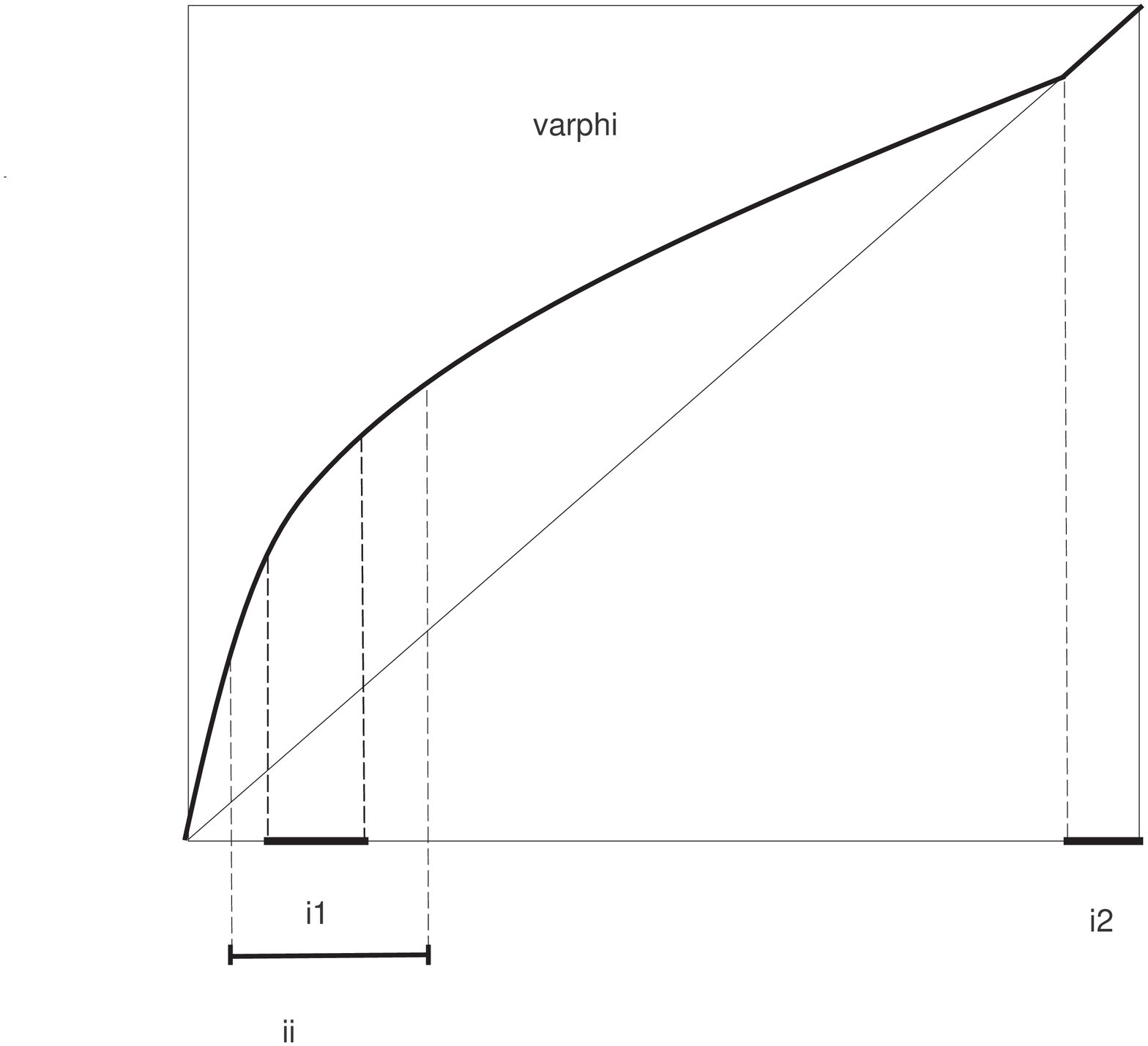}}
\subfigure[]{\includegraphics[scale=0.17]{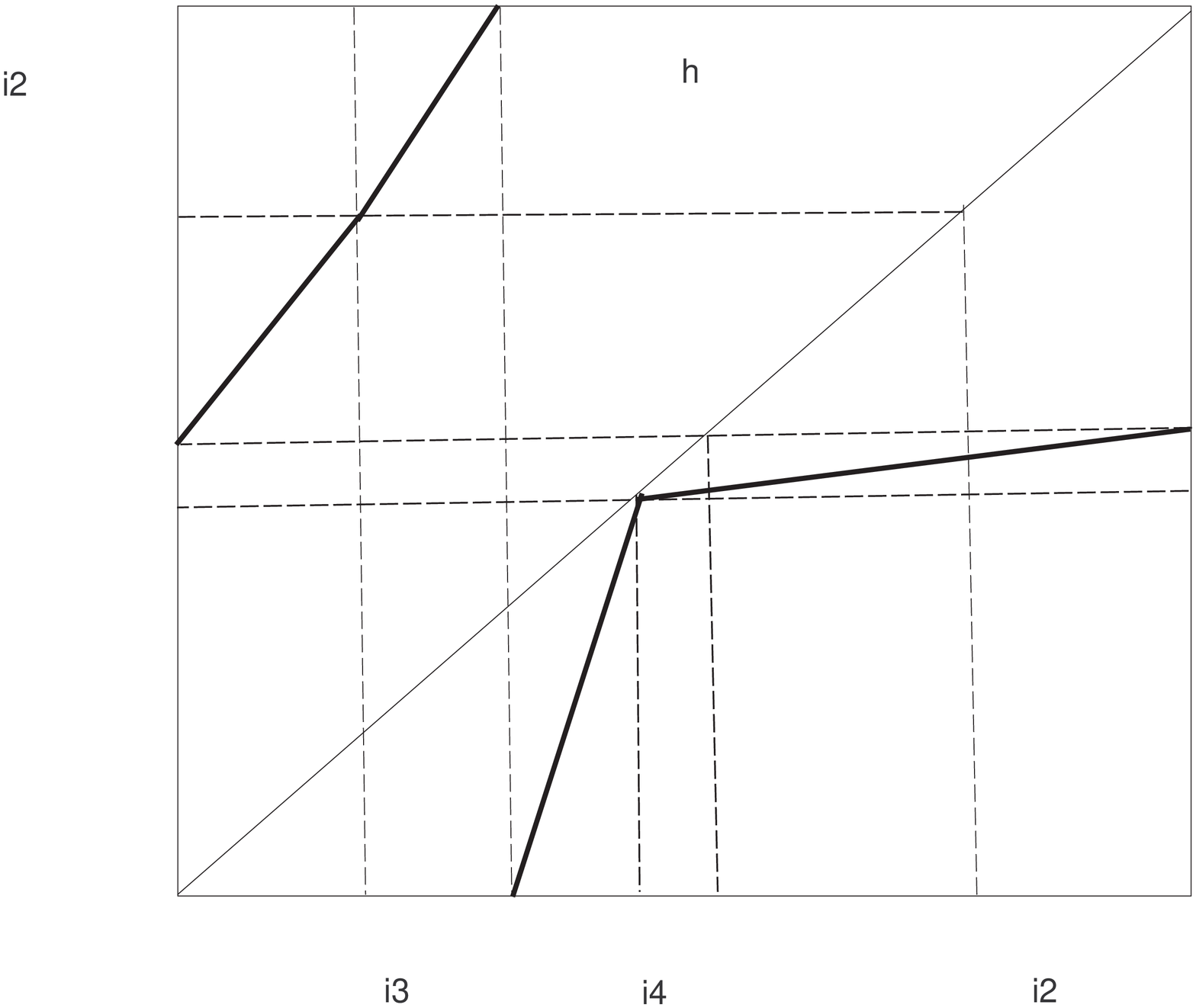}}
\end{center}
\end{figure}

We consider $$IFS( T^{+}_{I_{0}I_1},\ \ T^{-}_{I_{0}I_1} , \ H , \  \psi,  h_{J^{'},J,I_{0}},     ).  $$

We will prove that the minimal set for this IFS is  $$M=\bigcup_{n\geq 0}  \{   \psi^{n} (I_0) \} \cup I.     $$

We begin by proving that $M$ is IFS-invariant.
Let $x\in M$. If $x\in I_0$, then  $T^{+}_{I_{0}I_1}(x),  T^{-}_{I_{0}I_1}(x)\in I_0$, $H(x)\in I$ and as $J\cap I_0=\emptyset$ we have that $h_{J^{'},J,I_{0}}(x)\in I_0$. If $x\in  \psi^{n} (I_0) $ with $n>0$ as  $\psi^{n} (I_0)\cap I_1=\emptyset$ then $T^{+}_{I_{0}I_1}$ and $ T^{-}_{I_{0}I_1}$ are the identity on $\psi^{n} (I_0) $ with $n>0$. $H(x)\in \psi (I_0)$ and $J\cap \psi (I_0)=\emptyset$ then $h_{J^{'},J,I_{0}}(x)\in I_0$. The proof is analogous that in case $x\in I$.

 Now, we are going to prove that  any $x\in M$ has dense orbit in $M$. Let $y\in M$ and  $U$ a neighborhood  of $y$, we will show that there  exists $\phi \in IFS( T^{+}_{I_{0}I_1},\ \ T^{-}_{I_{0}I_1} , \ H , \  \psi,  h_{J^{'},J,I_{0}},     )  $ with $\phi (x)\in U$. First we take $x\in I_0$. If $y\in I_0$, it is clear that there exists $\phi \in IFS(  T^{+}_{I_{0}I_1},\ \ T^{-}_{I_{0}I_1} ) $ such that $\phi (x)\in U$.
 If $y\in  \psi^{n} (I_0)$, as $x\notin J$ then $ h_{J^{'},J,I_{0}}(x)\in I_{0}$, then there is $\phi_1\in IFS ( T^{+}_{I_{0}I_1},\ \ T^{-}_{I_{0}I_1})$ such that $\psi^{n}\circ\phi_1\circ h_{J^{'},J,I^{'}}(x)\in U$.
In the reaming cases the proof is analogous.

{\bf{Example 6.}} We consider $S^{1}$ as the interval $[0,1]$ identifying  $0$ with $1$.

 $M=\mathcal{I}\cup N$ where $\partial \mathcal{I}\cup \{0\} = N$. (countable union of intervals  union a finite number of points.

Let $\psi$ as in  Figure
 \ref{varphi} (b).  Let $I$ such that $\overline{I}\cap \psi (I)=\emptyset$, therefore it holds that $\psi ^{n}(I)\to 1$ $(1=0)$ . Let $J$ such that
$\overline {J}\cap \psi^{n} (I)=\emptyset$, for any $n\geq 0$ and let $J^{'}\subset J$. Consider $I^{'}\subset I$ and

$$IFS( \psi ,\ \  h_{J^{'},J,I^{'}},\ \ T^{+}_{I^{'}I},\ \ T^{-}_{I^{'}I}) .  $$
It is easy to show that

$$M= \{ 0\}\cup \bigcup _{n\geq 0}\psi^{n} (I^{'}). $$
{\bf{Example 7.}} $M=\mathcal{I}\cup K$ with $\partial \mathcal{I}\subset K$. (symmetric Cantorval).

Next lemma will be used in the construction of this example.
\begin{lemma}\label{l33}
Let $K_1\subset [a,b]$ and $K_2\subset [c,d]$ be  Cantor sets with $a,b\in K_1$ and $c,d\in K_2$ such that:

\begin{enumerate}
\item There exist sequences  $\{I_i\}_{i\in\N}$ and $\{I^{'}_i\}_{i\in\N}$, where $I_i$ and $ I^{'}_i$ are connected  components of $K_1^{c}$ with $(\cup I_i)\cup (\cup I^{'}_i) =K_{1}^{c}$ and $I_i\cap I'_k=\emptyset$ for all $i,k$.
\item There exist sequences  $\{J_j\}_{j\in\N}$ and $\{J^{'}_j\}_{j\in\N}$, where $J_j$ and $ J^{'}_j$ are connected  components of $K_2^{c}$ with $(\cup J_j)\cup (\cup J^{'}_j) =K_{2}^{c}$ and $J_j\cap J'_k=\emptyset$ for all $j,k$.
\item $\overline{  \{\partial I_i:\ i\in \N\}}=\overline{  \{\partial I^{'}_i:\ i\in \N\}}=K_1$ and $\overline{\{\partial J_j: \ j\in \N \}}=\overline{\{\partial J^{'}_j: \ j\in \N \}}=K_2$.

\end{enumerate}

Then there exists a increasing homeomorphism   $\psi:[a,b]\to [c,d]$ such that $\psi (K_1)=K_2$ and $\psi (\cup I_i)=\cup J_j$.

\end{lemma}

\begin{proof}

Let $I_{i_{0}}$ be an interval of  $\{I_n \}$ with maximal length and we choose $J_{j_{0}}$ in an analogous way.
Let $\psi :\overline{I}_{i_{0}}\to \overline{J}_{j_{0}}$ any increasing homeomorphism.
Among all the elements of $\{I'_n \}$ to the left of $I_{i_{0}}$, let $I'_{i_{1}}$ one of maximal length.

Among all the elements of $\{I'_n \}$ to the right of $I_{i_{0}}$, let $I'_{i_{2}}$ one of maximal length.

In an analogous way we define  $J^{'}_{j_{1}}$ and $J^{'}_{j_{2}}$. So define

 $\psi :\overline{I^{'}}_{i_{1}}\to \overline{J^{'}}_{j_{1}}$ and $\psi :\overline{I^{'}}_{i_{2}}\to \overline{J^{'}}_{j_{2}}$ as any increasing homeomorphism.

Among all the elements of  $\{I_n \}$ to the left of $I^{'}_{i_{1}}$ let $I^{}_{i_{3}}$ be one of maximal length. Let $I^{}_{i_{4}}$ be one of maximal length between $I^{'}_{i_{1}}$ and $I^{}_{i_{0}}$. let $I^{}_{i_{5}}$  be one of maximal length between $I^{}_{i_{0}}$ and $I^{'}_{i_{2}}$, and let $I^{}_{i_{6}}$ be
 be one of maximal length to the right of $I^{'}_{i_{2}}$. Analogously define $J_{j_{3}}$, $J_{j_{4}}$, $J_{j_{5}}$, $J_{j_{6}}$ and  $\psi :\overline{I}_{i_{p}}\to \overline{J}_{j_{p}}$ for $p=3,4,5,6.$

Proceeding inductively we obtain   $\psi: (\cup \overline{I}_i )\cup (\cup \overline{I^{'}}_i)    \to (\cup \overline{J}_j )\cup (\cup\overline{J^{'}}_j)$ an increasing map such that if $I$ and $I^{'}$ are connected components of $K_1^{c}$ such that $I$ is on the left of $I^{'}$ then $\psi (I)$ is on the left of $\psi(I^{'})$.

  Now we will prove that  $\psi: (\cup \overline{I}_i)\cup(  \cup \overline{I^{'}}_i)    \to (\cup \overline{J}_j) \cup (\cup \overline{J^{'}}_j)$ is uniformly continuous.
Suppose by contradiction that there exist  $\{x_n\}$,  $\{y_n\}$ and
$\varepsilon_0>0$ such that $d(x_n,y_n)<1/n$ and
$d(\psi(x_n),\psi(y_n))>\varepsilon_0$. Without loss of generality we can assume that
 $x_n<y_n$,  $x_n,y_n\to x_0\in [a,b]$ and  $x_n$ and $y_n$ are in different connected components of $ (\cup \overline{I}_i)\cup(  \cup \overline{I^{'}}_i)  $  . Since $\psi$ is continuous in $(\cup \overline{I}_i)\cup(  \cup \overline{I^{'}}_i) $, we have that $x_0 \notin (\cup \overline{I}_i)\cup(  \cup \overline{I^{'}}_i)$.

So,  if $n$ is big enough,  $x_n$ and $y_n$ are in different connected components of $K_1^c$. As
$d(\psi(x_n),\psi(y_n))>\varepsilon_0$ there exists $J=(\alpha, \beta)$ connected component
of $K_2^c$ such that $\psi(x_n)\leq\alpha<\beta\leq \psi(y_n)$,  as
$(\cup  J_j)\cup (\cup J^{'}_j)=K_2^c$ then there exists a connected components $I$ of
$K_1^c$ with $I\subset (\cup  I_n)\cup (\cup I^{'}_n)$ and $\psi(I)=J$. But as
 $x_n\leq \psi^{-1}(\alpha )<\psi^{-1}(\beta)\leq y_n$ and $x_n,y_n\to
x_0$ we obtain a contradiction.

As $\overline{  \{\partial I_i:\ i\in \N\}}=\overline{  \{\partial I^{'}_i:\ i\in \N\}}=K_1$ then we can defined  $\psi$ in  $[a,b]$ , by continuity and we are done.\end{proof}

\begin{figure}[h]
\psfrag{i0}{$I_{i_{0}}$}

\psfrag{i1}{$I^{'}_{i_{1}}$}

\psfrag{i2}{$I^{'}_{i_{2}}$}

\psfrag{i3}{$I_{i_{3}}$}

\psfrag{i4}{$I_{i_{4}}$}

\psfrag{i5}{$I_{i_{5}}$}

\psfrag{i6}{$I_{i_{6}}$}

\psfrag{t+}{$T^{+}$}

\psfrag{c}{$c$}
\psfrag{d}{$d$}
\psfrag{x0}{$x_0$}
\psfrag{h}{$h$}\psfrag{psi}{$\psi$}\psfrag{r1}{$r_1$}\psfrag{r2}{$r_2$}\psfrag{id}{$Id$}
\begin{center}
\caption{\label{figura7}}
\subfigure[]{\includegraphics[scale=0.15]{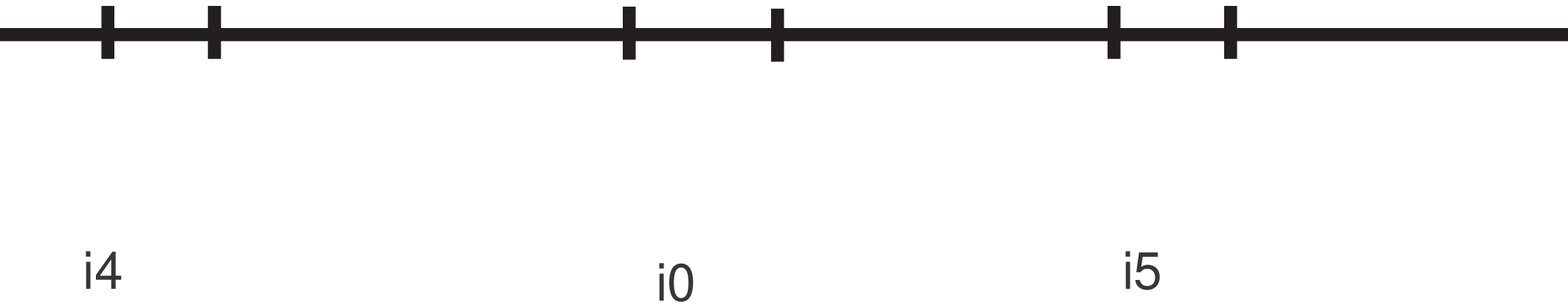}}
\end{center}
\end{figure}

Now we will  construct the example 7.

Given an interval $I\subset S^{1}$, we consider an interval  $I^{'}\subsetneq I$. Let $f,g$ and $h$ linear functions with angular coefficient  1/5 in $I^{'}$ (as Figure \ref{figura6} (a)) and such that $f|_{I^{c}}=g|_{I^{c}}=h|_{I^{c}}=Id$.
If $\Lambda_0=I^{'}$ and $\Lambda_{k+1}=f(\Lambda_k)\cup g(\Lambda_k)\cup h(\Lambda_k)$ then $K=\cap_{k\geq 0} \Lambda_k$ is a Cantor set.
\begin{rk}\label{rk1}

We make the following observations regarding the family  $\mathcal{F}_1=IFS(f, g,h) $:
\begin{itemize}
\item If $J\subset I^{'} $ is a connected component of $K^{c}$ then $J=\phi (I_0)$ or $J=\phi (I_1)$ with $\phi\in \mathcal{F}_1$, where $I_0$ and $I_1$ are as Figure \ref{figura6} (a).

\item If $J_0=\phi_0(I_0)$ and  $J_1=\phi_1(I_1)$ with $\phi_0,\phi_1 \in \mathcal{F}_1$ then $J_0\cap J_1=\emptyset$.

Therefore the set of connected components of $K^{c}$ is the union of $\{I_n\}$ and $\{J_n\}$, where $I_n=\phi_0(I_0)$ for some $\phi_0 \in \mathcal{F}_1$ and $J_n=\phi_1(I_1)$ for some $\phi_1 \in \mathcal{F}_1$. Moreover $\overline{  \{\partial I_n:\ n\in \N\}}=\overline{  \{\partial J_n:\ n\in \N\}}=K$.

\item The set $K\cup (\cup_{\phi \in {\mathcal{F}_1}}\phi (I_0))$ is not connected neither is a Cantor set.

\end{itemize}

\end{rk}

We consider an interval $[a,b]\subset I^{'}$ with $a,b\in K$ such that $I_0\subset [a,b]$, $[a,b]\cap I_1=\emptyset$ and $K\cap [a,b]$ is a Cantor set. Let $\overline{I_0}=[c,d]$ and $x_0\in K\cap (a,c)$ be such that $K\cap [a,x_0]$ is a Cantor set. By the lemma \ref{l33}, we can to construct $\psi :[a,b]\to [a,b]$ a increasing homeomorphism such that:

\begin{itemize}
\item $\psi (K\cap [a,x_0])=K\cap [a,c]$.
\item If $\mathcal{H}=\{ \phi (I_0):\ \phi\in\mathcal{F}_1  \}$ then $\psi (\mathcal{H} \cap [a,x_0])\subset \mathcal{H}\cap [a,c]$.
\end{itemize}

Let $T^{+}$ be such that  (see Figure\ref{figura6} (b)):

\[T^{+}(x)= \left\{
\begin{array}{l}
\psi (x) \mbox{ if $x\in [a,x_0],$}\\
r_{1} (x) \mbox{ if $x\in [x_0,c]$} \mbox{, with } r_1([x_0,c])\subset [c,d] \mbox { and  } r_1 \mbox{ linear}, \\
r_2 (x) \mbox{ if $x\in [c,d]$} \mbox { and  } r_2 \mbox{ linear}, \\
Id \mbox{ if $x\in [a,d]^{c}.$}\\
\end{array}
\right.\]

In summary, the map $T^{+}$ satisfies the following properties: points that are in  $K\cap [a,x_0]$ are mapped by $T^{+}$ in points that are in $K\cap [a,c]$. The elements of $\mathcal{H}$ included in $[a,x_0]$ are mapped by $T^{+}$ in the elements of $\mathcal{H}$ that are in $[a,c]$ and $T^{+}([x_0,d])\subset [c,d].$\\
If we consider $y_0\in (d,b)$ in analogous way to  $x_0$, we can to construct $T^{-}$ (in analogous way to construction of  $T^{+}$) such that: points that are in  $K\cap [y_0,b]$ are mapped by $T^{-}$ in point that are in $K\cap [d,b]$, the elements of  $\mathcal{H}$ included in $[y_0,b]$ are mapped by $T^{-}$ in the elements of $\mathcal{H}$ that are included in  $[d,b]$ and $T^{-}([c,y_0])\subset [c,d]$ and $T^{-}([c,b]^{c})=Id.$ \\
Moreover taking $r_1$ and $r_2$ conveniently in $T^{+}$ ( in analogous way  for $T^{-}$), we can construct $T^{+}$ and $T^{-}$ such that $I_0=[c,d]$ is a minimal set for the family  $IFS(T^{+}, T^{-} )$ (this construction is as the example of Subsection \ref{subsec2}).

\begin{figure}[h]
\psfrag{i0}{$I_0$}
\psfrag{t+}{$T^{+}$}
\psfrag{dd}{$I_0$}
\psfrag{i1}{$I_0$}
\psfrag{i2}{$I_{1}$}
\psfrag{f}{$f$}
\psfrag{g}{$g$}
\psfrag{i}{$\widehat{I}$}
\psfrag{iii}{$I^{'}$}
\psfrag{ii}{$I^{'}$}
\psfrag{a}{$a$}
\psfrag{b}{$b$}
\psfrag{c}{$c$}
\psfrag{d}{$d$}
\psfrag{x0}{$x_0$}
\psfrag{y0}{$y_0$}

\psfrag{h}{$h$}\psfrag{psi}{$\psi$}\psfrag{r1}{$r_1$}\psfrag{r2}{$r_2$}\psfrag{id}{$Id$}
\begin{center}
\caption{\label{figura6}}
\subfigure[]{\includegraphics[scale=0.13]{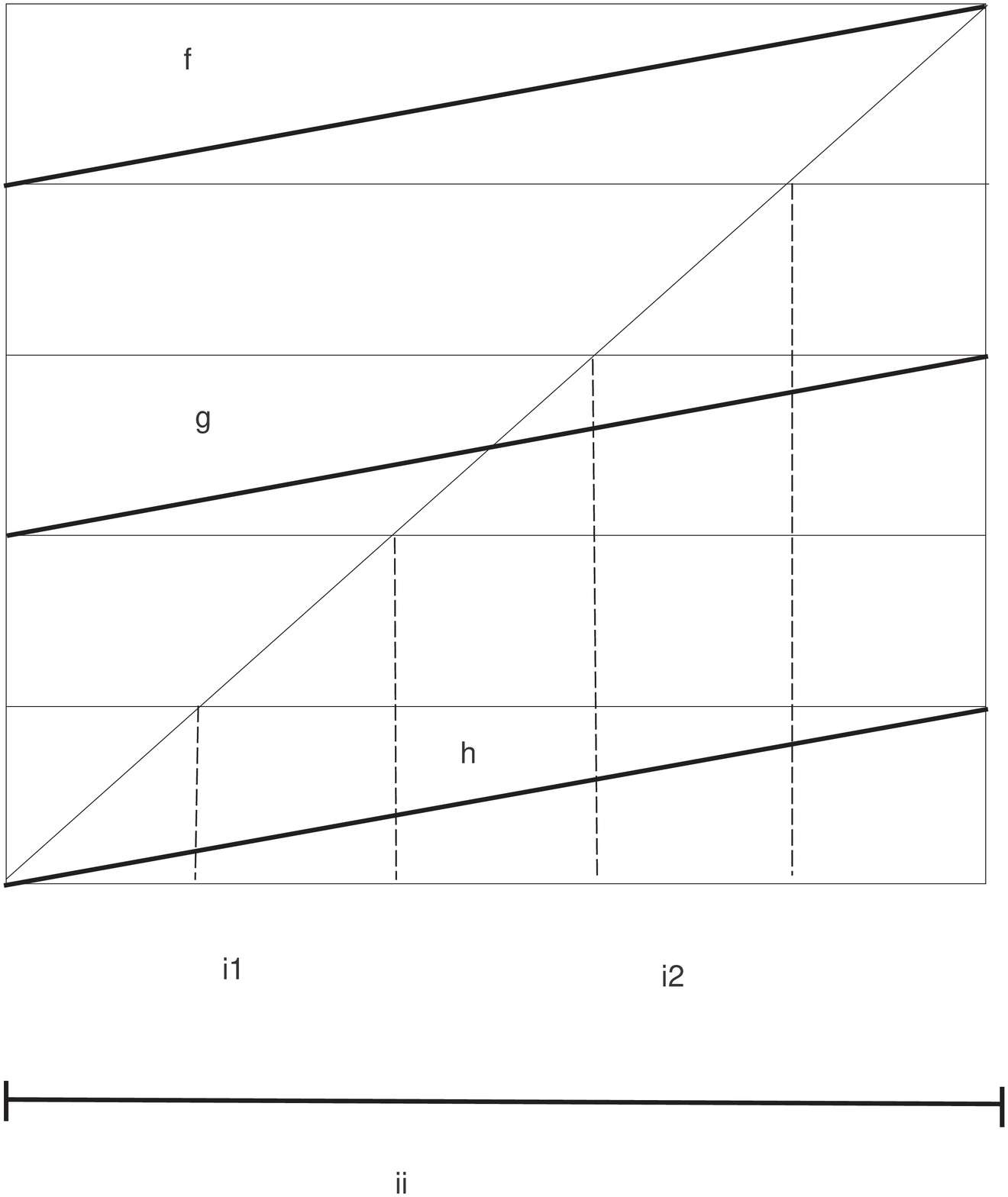}}
\subfigure[]{\includegraphics[scale=0.13]{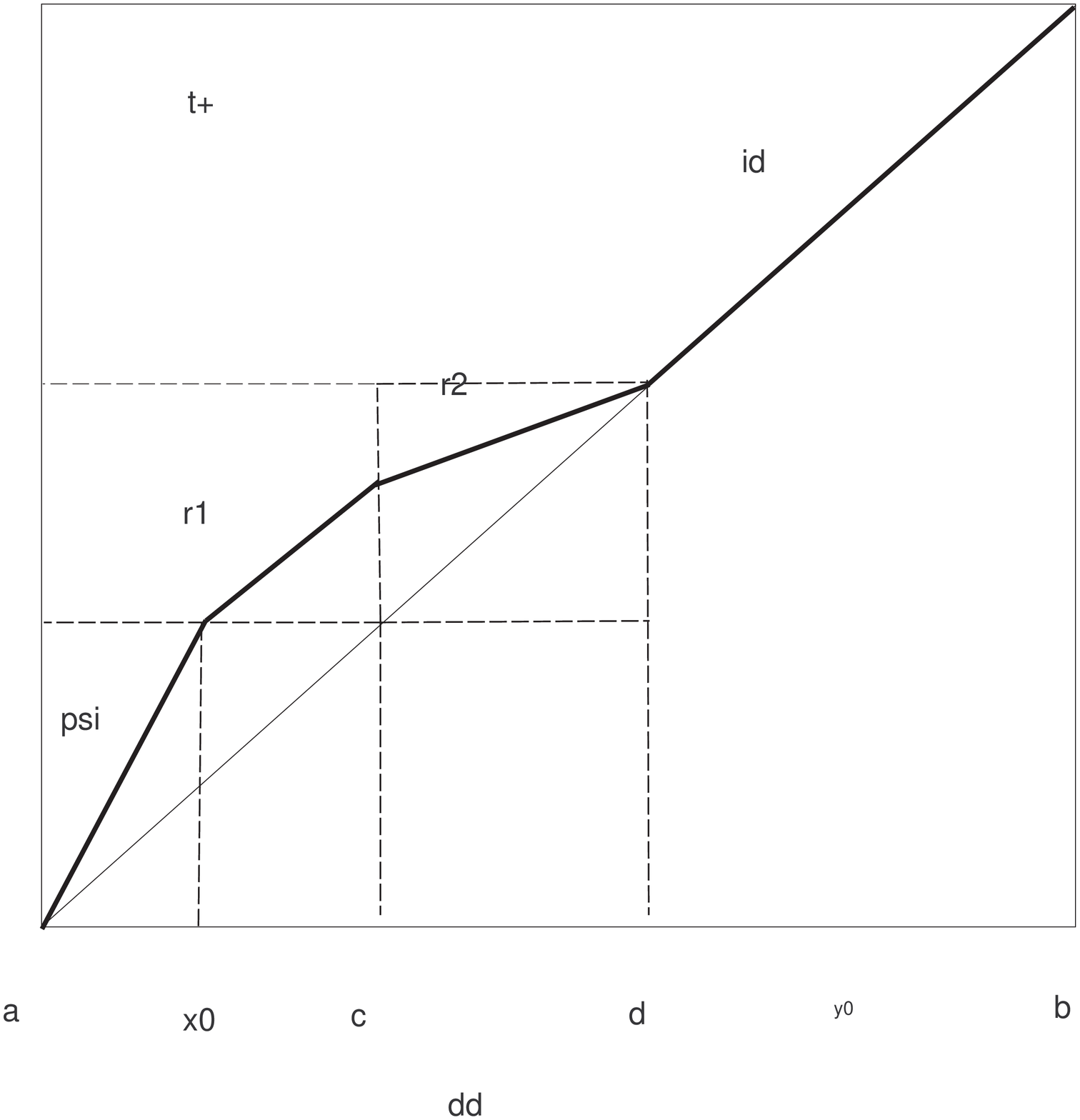}}
\end{center}
\end{figure}

Let $IFS(f,g,h, T^{+},T^{-})$ and $M=K\cup  \{   \phi (I_{0}): \ \ \phi \in IFS(        f , g, h          \} \cup I_{0} . $

The proof of minimality of $M$ is analogous of above examples. $\Box$

\end{document}